\documentclass[11pt,leqno]{amsart}
\pdfoutput=1
\usepackage{enumerate,amsthm,amsmath,amssymb,amsxtra,comment,graphicx,psfrag}
\usepackage{xspace}
\usepackage{pdfsync}
\usepackage{color}
\usepackage{times}
\usepackage{enumerate,amsmath,amssymb,amsxtra,graphicx,psfrag,a4wide}
\usepackage{xspace}
\usepackage{pdfsync}
\usepackage{color}
\usepackage{mathrsfs}
\usepackage{amssymb}
\usepackage{fancybox}
\usepackage{subfigure}
\usepackage{enumerate, verbatim}
\usepackage[normalem]{ulem}

\def\Lat{\mathfrak{L}}

\def\Y{{\mathcal{X}}}

\def\V{{\mathscr{V}}}

\def\F{\mathsf{F}}
\def\T{{\mathcal{T}}}
\def\Vh{\boldsymbol{V}_{\varepsilon, Q}}
\def\Wh{\boldsymbol{V}_{\varepsilon, T}}
\def\Vac{\boldsymbol{V}_{h, ac}}

\def\<{{\langle }}
\def\>{{\rangle }}

\def\be{\beta}

\def\R+{{\Bbb R}_*^+}

\newcommand{\DD}{\overline{D}}

\newcommand{\J}[1]{{  [\! [#1]\! ]  }}   

\newcommand{\Av}[1]{{  \{\!\!\{#1\}\!\!\}  }}

\newcommand{\ppc}[1]{\mbox{\kern-3pt\b{\kern3pt$#1$}}}

\newcommand{\dd}{\,{\rm d}}

\newtheorem{lemma}{Lemma}[section]
\newtheorem{proposition}{Proposition}[section]

\newtheorem{remark}[lemma]{Remark}

\vspace*{2.0cm}

\renewcommand{\be}{\begin{equation}}
\newcommand{\ee}{\end{equation}}
\newcommand{\ph}{\phi_\eta }

\newcommand{\Om}{\Omega}
\newcommand{\ZZ}{\mathbb{Z}}
\newcommand{\RR}{\mathbb{R}}

\newcommand{\GG}{\Gamma}

\def\XXint#1#2#3{{\setbox0=\hbox{$#1{#2#3}{\int}$}
     \vcenter{\hbox{$#2#3$}}\kern-.5\wd0}}

\numberwithin{equation}{section}

\begin{document}
\title{ On  atomistic-to-continuum couplings without ghost  forces in three dimensions}
\author{Charalambos Makridakis, Dimitrios Mitsoudis and Phoebus Rosakis 
}                     
%
%
\maketitle
\begin{abstract}
In this paper we  construct  energy based  numerical methods free of ghost forces in \emph{three dimensional lattices} arising in crystalline materials. The analysis hinges  on establishing a connection 
of the coupled system to conforming finite elements.  Key ingredients are: (i) a new representation of
discrete derivatives related  to  long range interactions of atoms as  volume integrals of
gradients of piecewise linear functions over \emph{bond volumes},  and (ii)  the construction of an underlying 
globally continuous function representing the coupled modeling method.  
%
\end{abstract}



\section{Introduction}
\label{intro}
In recent years substantial progress has been made  in multiscale modeling   of materials, see e.g.,  \cite{BLeBL07,KaxL05}. A class of important  problems concerns  atomistic-to-continuum coupling in crystals, e.g., the quasicontinuum method  \cite{SMT99} and its variants. Since the continuum model fails to provide an accurate prediction in the vicinity of defects and other singularities,  coupled 
atomistic/continuum methods have become popular as an adaptive modeling approach over the 
last years, see e.g. references in   \cite{MS11, Shap10,Shap11_3D}. The main issue that arises in these methods is the proper matching of information across scales. In the first attempts in this direction,  ad hoc coupling of atomistic and continuum energies resulted in numerical artifacts at the interface of atomistic and continuum 
regions, known as \emph{ghost forces,} e.g., \cite{DL09ii}.
Therefore,  the construction of \emph{consistent} A/C couplings (that are free of ghost forces) is crucial   in the numerical modeling of crystalline materials. Further, since this problem is one of the 
better identified mathematical problems related to matching of information across scales in materials, it might provide useful insight into the study of  multi-scale
computational methods of a more general nature.   

This paper is devoted to the construction of energy based methods free of ghost forces in \emph{three dimensional crystal
lattices.} The problem of constructing consistent energies in two dimensional lattices 
was resolved recently by Shapeev \cite{Shap10}, see also \cite{LiLuskin2012}. A key idea in \cite{Shap10} is to express differences (discrete derivatives) related to long range interactions of atoms as  appropriate line integrals over
bonds. In two space dimensions it is then possible to transform the assembly of line
integrals over all possible interactions into an area integral,  through a counting argument
known as  the \emph{bond density lemma}, \cite{Shap10}. This lemma fails to hold in three space dimensions,
 thus the construction of energy based consistent couplings based on this approach 
does not seem to be readily extendable to this case; see  \cite{Shap11_3D} where an interesting attempt to circumvent this problem is made.  {Other papers dealing with  similar problems include, e.g., 
 \cite{Belytschko2002, Shimokawa2004, Eetall2006, LinSh10, OrtZhang2011}.
}

Our work adopts a different approach, based on control volumes associated with bonds, 
which we call \emph{bond volumes,} and  on the construction of an underlying globally continuous function representing the coupled modeling method. 
The three dimensional coupled energies constructed in this way are free of ghost forces. Moreover, they can be combined in a consistent way to high-order finite element discretizations  of the continuum region. 

\newpage
The paper is organized as follows. In Section 1 we introduce  necessary notation. In  Section 2 we introduce suitable finite element spaces and  atomistic Cauchy-Born models
which are used in the construction of the coupled methods. 
In Section 3 we state and prove a key result, Lemma 3.1, which establishes a connection between long range differences and volume integrals of piecewise linear functions defined over appropriate decompositions of bond volumes   into tetrahedra. In Section 4 we present a conforming coupling  method based on bond volumes. 
We note that in the continuum region we use the atomistic Cauchy-Born models,  introduced in \cite{MS11}.  In Section 5
%
we show that it is possible 
 to introduce discontinuities at the interface, thus allowing greater 
flexibility in the design of underlying meshes, while  still 
obtaining a consistent ghost-force-free method. The analysis in this section may 
lead to the design of more general atomistic/continuum coupled methods based on 
discontinuous finite elements. Finally, in Section 6 we show that one can use finite elements
of high order to discretize  the continuum region. All methods  presented here are free of ghost forces;
they provide a framework that facilitates  the design of several  alternative formulations.

\subsection{Notation}
\emph{Lattice, discrete domain, continuum domain. \/} {We let $e_i$ be the standard basis vectors for $\RR^3$, and choose $\mathbb{Z}^3$ as the three-dimensional lattice. The extension  to  lattices generated by any three linearly independent vectors  of $ \mathbb{R}^3 $  is straightforward since it merely  involves   compositions
 with a fixed
affine map.  The \emph{scaled lattice} is $\varepsilon \mathbb{Z}^3=\{x_ {\boldsymbol{\ell} }= (x_{\ell _1},x_{\ell _2}, x_{\ell _3})= \varepsilon\, \ell  ,\,\ell\in \mathbb{Z}^3 \}$, wth  lattice distance $\varepsilon=1/k$,  $k\in\mathbb{Z}_+$.
We will consider discrete periodic functions on $\mathbb{Z}^3$ defined over a `periodic domain' $\Lat$. 
More precisely, let  $M_i\in \mathbb{Z}_+$, $i=1,2,3$ and define 
\begin{equation*}
\begin{split}
&  \Omega:= [-M_1+1, M_1]\times [-M_2+1, M_2] \times [-M_3+1, M_3] .\\
&\Omega_{\text{discr}} :=\varepsilon \mathbb{Z}^3\cap \Omega, \quad \Lat := \mathbb{Z}^3\cap\frac{1}{\varepsilon} \Omega.
\end{split}
\end{equation*}
Here $\Omega$ is the  \emph{continuum domain};  the actual configuration of the atoms is $\Omega_{\text{discr}}$, the set of atoms of the scaled lattice contained in $\Omega$.  In particular, the convex hull of   $\Omega_{\text{discr}}$ is $\Omega$.
Also $\Lat$ is the basic lattice period in the unscaled lattice $\ZZ^3$. }

\noindent
\emph{Functions and spaces.\/} The atomistic deformations are denoted 
\begin{equation*}
\begin{split}
&y_\ell= y( x_ \ell)  \, , \; x_\ell=\varepsilon\ell, \quad \ell \in \Lat \quad \text{where}\\
&y_\ell= \F  x_ \ell + v_\ell, \quad \text{with $ v_\ell = v ( x_ \ell)$ periodic with respect to $\Lat$}.
\end{split}
\end{equation*}
Here $\F $ is a constant $3\times 3$ matrix with $\text{det}\,  \F  >0$.
The corresponding spaces for $y$ and $v$ are denoted by
$ \Y $ and $\V$ and are defined as follows:
\begin{equation*}
\begin{split}
& {\Y} := \{  {y} : \Lat \to \mathbb{R}^3, \quad  y_\ell= \F  x_ \ell + v_\ell,  \quad \ v \in \V,\quad \ell \in \Lat \} \, ,\\
&\V := \{ u : \Lat \to \mathbb{R}^3,    \quad \text{$ u_\ell = u ( x_ \ell)$ \quad  periodic with zero average with respect to $\Lat$} \}.
\end{split}
\end{equation*}
For functions ${y} , v\, : \Lat \to \mathbb{R}^3$ we define the inner product
$$\<\,  y, v\, \> _\varepsilon    := \varepsilon ^3 \sum _{\ell \in \Lat } \ \, \ y_\ell \, \cdot \, v_\ell.$$

\noindent
For a positive real number $s$ and $1 \leq p \leq \infty$
we denote by $W^{s,p} (\Omega , \mathbb{R}^3) $ the usual Sobolev space of functions $ y: \Omega \to \mathbb{R}^3.$
By $W_{\#}^{s,p} (\Omega , \mathbb{R}^3)$ we denote the corresponding
Sobolev space of periodic functions with basic period $\Omega$. By $\<\,  \cdot ,  \cdot \, \>$ we denote the standard $L^2 (\Omega)$ inner product;
 {{for a given nonlinear operator $A,$    we shall denote as well  by $\<\,  D A,  v \, \>$   the action of its 
 derivative  $D A$ as a linear operator applied to $v. $}} The space corresponding
to $\Y$ in which the minimizers of the continuum problem are sought is
\begin{equation*}
\begin{split}
&X:= \{  {y} : \Omega \to \mathbb{R}^3, \quad  y (x) = \F  x  + v (x),  \quad \ v \in V\},\qquad \mbox{where}\\
&V:= \{ u : \Omega \to \mathbb{R}^3, \quad      u\in W^{k,p} (\Omega , \mathbb{R}^3)  \cap  W_{\#}^{1,p} (\Omega , \mathbb{R}^3), \ \quad   \int _\Omega u \dd x = 0\}.
\end{split}
\end{equation*}
%

\smallskip
\noindent
\emph{Difference quotients and derivatives.\/} The following notation
will be used throughout:
\begin{equation}
\label{DD}
\DD_{\eta} y_{\ell}:=\frac{y_{\ell+\eta} -y_{\ell} }{\varepsilon}, \qquad  \ell, \ \ell+\eta \in \Lat,
\end{equation}
denotes the difference quotient (discrete derivative) in the direction of the vector $\eta.$ Also,  %
\begin{align}
\begin{split}
{\partial}_{\zeta_i} \phi ( \zeta) &:= \frac{\partial \phi  (\zeta_1, \zeta_2, \zeta _3) } {\partial  {\zeta_i}}  \, , \qquad \zeta= (\zeta_1, \zeta_2, \zeta _3),\\
{\nabla}_{\zeta} \phi ( \zeta) &:= \Big \{{\partial}_{\zeta_i} \phi ( \zeta) \Big  \} _{i },\\
{\partial}_{\alpha }v(x) &:=\frac{\partial v (x)} {\partial {x_ \alpha}}, \quad
{\nabla} u (x) :=\Big \{\frac{\partial u^i  (x) } {\partial  {x_\alpha}} \Big  \} _{i\alpha}.
\end{split}
\end{align}
%

\smallskip
\noindent
\emph{Atomistic and Cauchy--Born potential. \/}
We consider the atomistic  {energy}
\begin{equation}
\label{1.8}
{\Phi}^ a (y) := \varepsilon ^3 \sum _{\ell \in \Lat } \ \, \sum _{\eta \in R}\  \phi _\eta \, (\DD_{\eta} y_{\ell} ),
\end{equation}
where $R\subset\mathbb{Z}^3$ is a given finite  set of interaction vectors, and 
the interatomic  potential $ \phi _\eta (\cdot)$ may vary with the type of bond, i.e.,
$\phi _\eta$ may depend explicitly on  $\eta$.   Further, $\phi _\eta (\cdot)$
is assumed to be sufficiently smooth. 

For a given field of external forces $f : \Lat \to \mathbb{R}^3, $ where $ f_\ell = f( x_ \ell),$ the atomistic problem reads as follows:
\begin{equation}
\label{atomisticModel}
\begin{split}
  & \text{find a local minimizer $y^a$  in $\mathcal{X} $ of : }\\
 &\ \  {\Phi}^ a (y) - \< f , y \> _\varepsilon.
\end{split}
\end{equation}
If such a minimizer exists, then
\begin{equation*}
   \< D{\Phi}^ a (y^a), v\> _\varepsilon =   \< f , v \> _\varepsilon   \, , \qquad \text{ for all } v\in \V,
\end{equation*}
where
\begin{equation}
\begin{split}\label{D-atom}
\< D{\Phi}^ a (y), v\> _\varepsilon   &:= \varepsilon ^3 \sum _{\ell \in \Lat } \ \, \sum _{\eta \in R} \ \,   \partial _{\zeta_i}   \phi _\eta \, (\DD_{\eta} y_{\ell} )  \,   \big [  \DD_{\eta} v_{\ell}   \big ]    _i\\
 &\ = \varepsilon ^3 \sum _{\ell \in \Lat } \ \, \sum _{\eta \in R}\ \nabla _{\zeta}   \phi _\eta \, (\DD_{\eta} y_{\ell} )    \, \cdot \, \DD_{\eta} v_{\ell}.
\end{split}
\end{equation}
We employ  the summation convention for repeated indices.\\

The corresponding Cauchy--Born stored energy function is  \cite{Er84,BLeBL02},
\begin{equation*}
W (\F ) =W_{CB} (\F ) :=  \, \sum _{\eta \in R}\  \phi _\eta \, (\F\, \eta ).
\end{equation*}
Then, the continuum Cauchy--Born model is stated as follows:
\begin{equation}
 \label{CBmodel}
\begin{split}
  & \text{find a local minimizer $y^{CB}$  in $X$ of : }\\
 &\ \  {\Phi}^ {CB} (y   ) - \< f , y \>,
\end{split}
\end{equation}
where the external forces $f $ are appropriately related to the discrete external forces and
\begin{equation*}
\ \  {\Phi}^ {CB} (y) : = \int_\Omega \,   W_{CB} (\nabla  y (x)  )\dd x.
\end{equation*}
 If such a minimizer exists,  (and is a   diffeomorphism on $\Om$) then
\begin{equation}
\label{CB-variational}
   \< D {\Phi}^ {CB} (y^{CB})  , v\>  =   \< f , v \>     \, , \qquad \text{ for all } v\in  V,
\end{equation}
where
\begin{equation*}
\begin{split}
   \< D{\Phi}^ {CB} (y), v\>  =      \int_\Omega  S_{i\alpha} (\nabla y(x))   \, \frac{\partial v^i  (x) } {\partial  {x_\alpha}}\dd x
   =      \int_\Omega  S_{i\alpha} (\nabla y(x))   \,  \partial_\alpha  v^i  (x) \dd x \, , \quad  \,  v\in V .
\end{split}
\end{equation*}
Here the stress  tensor $S$  is defined, as usual, by
\begin{equation*}
\begin{split}
&S  := \Big \{\frac{\partial W (\F ) } {\partial  {F_{i\alpha}}} \Big  \} _{i\alpha}.
\end{split}
\end{equation*}
The stress tensor and the atomistic potential are related through:
\begin{alignat}{2}
\begin{split}
S_{i\alpha}  (\F )  &= \frac{\partial W (\F ) } {\partial  {F_{i\alpha}}}  
   = \sum _{\eta \in R}\  \partial _{\zeta_i} \phi _\eta \, (\F\, \eta ) \, \eta_\alpha .
\end{split}
\end{alignat}

\section{Finite element spaces and atomistic Cauchy--Born models}\label{sec2}
In the sequel we introduce the finite element spaces used in the rest of the paper.
In addition we  introduce an intermediate model connecting   the continuum  and  atomistic models.  We call this the \emph{atomistic Cauchy--Born model (A-CB).}

\smallskip

\smallskip

\emph{Trilinear finite elements on the lattice. \/} Let $\Vh$ be the linear space of all periodic functions  that are continuous and piecewise trilinear on $\Omega$. More precisely, let
\begin{equation*}
\begin{split}
&\T_Q:= \{ K \subset \Omega \, : \quad K = (x_{\ell _1}\, , x_{\ell _1+1})\times (x_{\ell _2}\, , x_{\ell _2+1})\times (x_{\ell _3}, x_{\ell _3+1}),\quad x_ {\boldsymbol{\ell} }\, = (x_{\ell _1}\, , x_{\ell _2} ,x_{\ell _3} ) \in \Omega_{\text{discr}} \}, \\
&\Vh := \{ v  : \Omega \to \mathbb{R}^3,    \quad \text{$v\in C(\Omega)\, ,\ v | _{K} \in \mathbb{Q}_1(K)$ \ and  $v_\ell = v ( x_ \ell)$ \quad  periodic with respect to $\Lat$} \},
\end{split}
\end{equation*}
where $\mathbb{Q}_1(K) $ denotes the set of all trilinear functions on $K.$ Whenever we wish to emphasize that we work on the specific cell $K= (x_{\ell _1}\, , x_{\ell _1+1})\times (x_{\ell _2}\, , x_{\ell _2+1})\times (x_{\ell _3}, x_{\ell _3+1})$ we shall denote it by $K _\ell\, .$ 
The elements of $\Vh $ can be expressed in terms of the nodal basis functions $\Psi _\ell =\Psi _\ell (x)$ as
\begin{equation*}
\begin{split}
 v(x)  
	=&\, \sum _ {\ell \in \Lat }   v _ \ell \, \Psi _{\ell_1} (x_1)\, \Psi _{\ell_2} (x_2)\, \Psi _{\ell_3} (x_3), \quad v_\ell = v ( x_ \ell),
\end{split}
\end{equation*}
where we have used the fact that $\Psi _\ell (x)$ can be written as the tensor product of the
standard one-dimensional piecewise linear hat functions $\Psi _{\ell_i} (x_i).$    Here   $ \Psi _{\ell_i} (x_{\tilde \ell _i}) = \delta _{\ell_i \tilde \ell _i}.$

\smallskip
\noindent
For any  connected set  $\mathcal{O} $ such that  
\be\overline  {\mathcal{O}}  = \overline{ \cup _{K \in \mathcal{S} _Q} K } ,  \ee
 $\mathcal{S} _Q$ being a subset of
$\T _Q$ we denote  by $\Vh (\mathcal{O} ) $ the natural restriction of  $ \Vh $ on  $ \mathcal{O}.$

\smallskip
\noindent
 \emph{Linear finite elements on {lattice tetrahedra}.    \/} Let $\Wh $ be the space of continuous periodic functions that are piecewise linear on  
lattice tetrahedra. A crucial observation is that there are more than one 
ways to subdivide   a given lattice cell $K$  into lattice tetrahedra. Our analysis is sensitive to the choice of  such a subdivision.
At this point we assume that the lattice tetrahedra   in the following 
definition are all of the same type, i.e., they have been obtained via the same type of subdivision of 
each lattice cell. With this in mind, 
we define
\begin{equation}  \label{Vet}
\begin{split}
&\T _T= \{ T \subset \Omega \, : \quad T \text{ is a tetrahedron  whose vertices are lattice vertices of }  K_\ell  \, ,\quad x_ {\boldsymbol{\ell} }\, 
\in \Omega_{\text{discr}} \}, \\
&\Wh  := \{ v  : \Omega \to \mathbb{R}^2,    \quad \text{$v\in C(\Omega)\, ,\ v | _{T} \in \mathbb{P}_1(T)$ \ and  $v_\ell = v ( x_ \ell)$ \quad  periodic with respect to $\Lat$} \},
\end{split}
\end{equation}
where $\mathbb{P}_1(T) $ denotes the set of affine functions on $T.$
As above, for any   connected set  $\mathcal{O} $ such that  
$\overline  {\mathcal{O}}  = \overline{ \cup _{T \in \mathcal{S} _T} T } $,
 $\mathcal{S} _T$ being a subset of
$\T _T$, we denote  by $\Wh (\mathcal{O} ) $ the natural restriction of  $ \Wh $ on  $ \mathcal{O}.$

\smallskip

\subsection{
Atomistic Cauchy--Born models on cells and  tetrahedra.}  A  decomposition of the 
cell $K_\ell $ with a vertex at $x_\ell$ into six tetrahedra is called a \emph{type A decomposition} if   the  diagonals $(x_{\ell} , x _{\ell+e_1+e_3})$
 and $(x_{\ell+e_2} , x _{\ell+e_1+e_2 +e_3})$ are edges of the resulting tetrahedra, see Fig. 1. In other words, the main diagonal $(x_{\ell} , x _{\ell+e_1+e_2+e_3})$, the three face diagonals starting at $x_{\ell} $,  the three face diagonals starting at $ x _{\ell+e_1+e_2+e_3} $, and the edges of  $K_\ell $, together  comprise the edges of the six tetrahedra.
%
%
%
\begin{figure}[h]
    \begin{center}
      \includegraphics[width=0.4\textwidth]{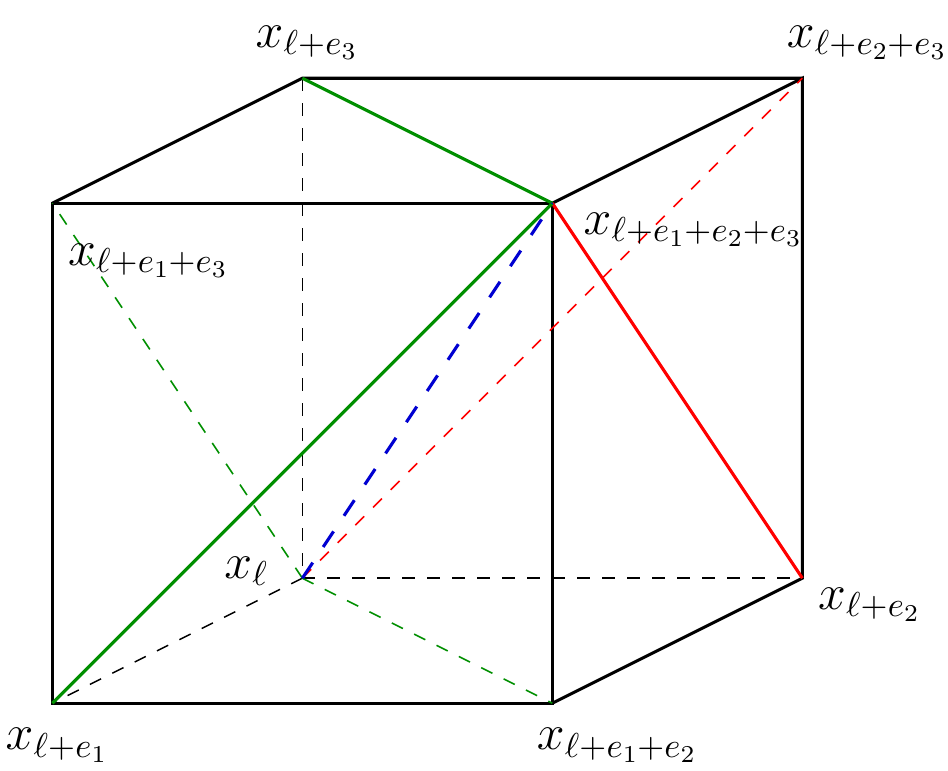}
      \caption{ A \emph{type A decomposition}  of the 
cell $K_\ell $   into six tetrahedra.}
    \label{Fig:1}
    \end{center}
\end{figure}
%
%
%
%
%
 Notice that  in each tetrahedron originating  from a type A decomposition of a cell,
exactly three edges are edges of the original cell,  {these are depicted
with solid, black lines in Fig.~2}.  To define the atomistic Cauchy--Born model on tetrahedra we need to define first discrete 
 gradients at each tetrahedron $T.$ To this end, we assume that all cells are divided into tetrahedra 
 from a type A decomposition. Let $v\in \Wh$.
  {Define $\widetilde \nabla v$  as
\begin{equation}\label{DT}
 \Big \{ \widetilde {\nabla} v |_{T}  \Big \} _{i\alpha}  :=   \widetilde {{D}}_{e_\alpha} v_{\ell} ^i ,
 \end{equation}
where the discrete derivatives $ \widetilde {{D}}_{e_\alpha} v_{\ell} ^i $ on the tetrahedron 
$T$    are just the difference quotients of $v$ along the edges of $T$    with directions  $e_\alpha$. These are the edges shared with those of $K_\ell$, shown in black solid lines in Fig.~2.  For example, for the tetrahedron of Fig.~2, 
$\widetilde {{D}}_{e_3} v_{\ell} ^i =\DD_{e_3} v_{\ell} ^i$, see \eqref{DD}, whereas  
$\widetilde {{D}}_{e_2} v_{\ell} ^i =\DD_{e_2} v_{\ell+e_1+e_3} ^i$.  }
Notice that the definition of these discrete derivatives can be extended to any smooth function.
 Then for each tetrahedron $T$  it follows that 
$$\int _T W_{CB} (\nabla v) dx = \frac {\varepsilon  ^ 3} 6 W_{CB} (\widetilde \nabla v). $$
  
 %
%
%
%
 %
%
\begin{figure}
    \begin{center}
      \includegraphics[width=0.45\textwidth]{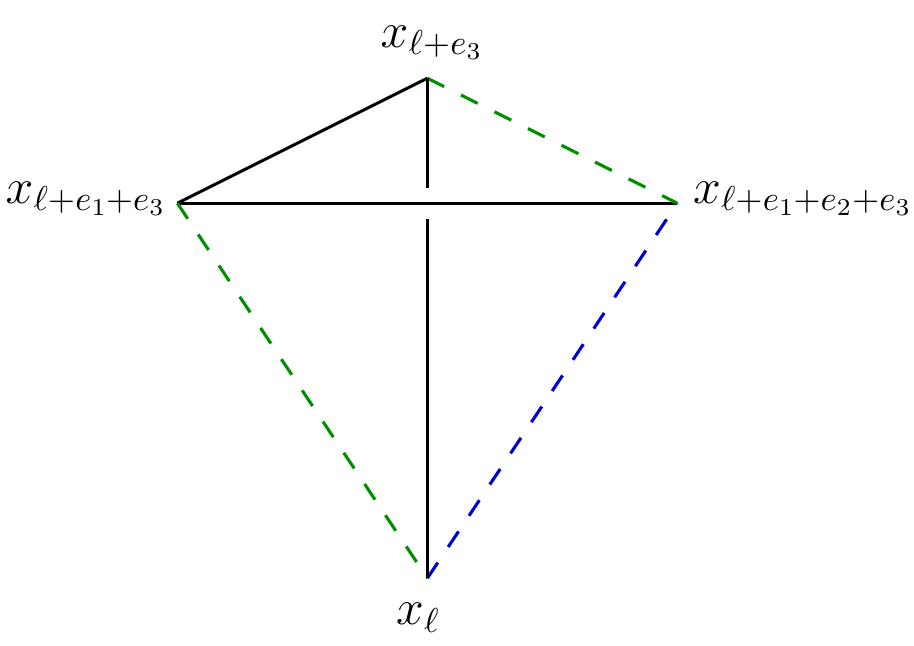}
       {
      \caption{A typical tetrahedron resulting from the decomposition of the cell $K_\ell$. 
      The three edges shown in solid, black lines are also edges of the original cell .}}
    \label{Fig:1d}
    \end{center}
\end{figure}

%
%
 %
%
Further,  let $y$
be a sufficiently smooth deformation. 
We define  corresponding the \emph{atomistic Cauchy--Born} (A--CB)  {energy}
\begin{equation}\label{ACB-tetrah}
\begin{split}
\tilde {\Phi}^ {a, CB} (y) &:= \frac {\varepsilon  ^ 3} 6  \sum _{\ell \in \Lat } \ \,   \sum _{T \in K_\ell(T) }   \sum _{\eta \in R}\  \phi _\eta \, (\widetilde {\nabla} y\,  \eta )
 = \frac {\varepsilon  ^ 3} 6 \sum _{\ell \in \Lat } \,   \sum _{T \in K_\ell(T) }    W_{CB} ( \widetilde {\nabla}y    ).
 \end{split}
\end{equation}
Now, for a given field of external forces $f : \Lat \to \mathbb{R}^3$ the \emph{{tetrahedral} A--CB problem}    reads as follows:
\begin{equation*}
\begin{split}
  & \text{find a local minimizer $y^{a, CB}$  in $\mathcal{X} $ of : }\\
 &\ \ \tilde  {\Phi}^ {a, CB} (y^{a, CB}) - \< f , y^{a, CB} \> _\varepsilon   .
\end{split}
\end{equation*}
If such a minimizer exists, then
\begin{equation*}
\begin{split}
   \< D \tilde  {\Phi}^ {a, CB} (y^{a, CB}) , v\> _\varepsilon =   \< f , v \> _\varepsilon   , \qquad \text{ for all } v\in \V.
\end{split}
\end{equation*}
This atomistic model is \emph{consistent}, in the sense that the above is satisfied for \emph{homogeneous} deformations ($y_F (x) = \F x$, $x\in\Omega$):
\be
 \< D \tilde  {\Phi}^ {a, CB}  (y_F), v\>=  0, \quad y_F (x) = \F x,
  \ee
for all   $v\in \Wh.$   
To show that, it suffices to observe
\be \begin{split}
\< D\tilde  {\Phi}^ {a, CB}  (y_F), v\>=&
\frac {\varepsilon  ^ 3} 6  \sum _{\ell \in \Lat } \ \,   \sum _{T \in K_\ell(T) }   
\sum _{\eta \in R}\  \phi _\eta \, (\widetilde {\nabla} y_F \, \eta )\cdot  \widetilde {\nabla} v\,  \eta\\
=&\sum _{\eta \in R}\ 
\ph'(F\,\eta)\cdot
\sum _{\ell \in \Lat } \ \,   \frac {\varepsilon  ^ 3} 6 \sum _{T \in K_\ell(T) }     \widetilde {\nabla} v\,  \eta\\
=&\sum _{\eta \in R}\ 
\ph'(F\,\eta)\cdot
\sum _{\ell \in \Lat } \ \,     \sum _{T \in K_\ell(T) }   \int _T    {\nabla} v\,  \eta \, dx\\
=&\sum _{\eta \in R}\ 
\ph'(F\,\eta)\cdot
 \int _\Omega    {\nabla} v\,  \eta \, dx= 0  . 
 \end{split}\ee
%
%
%
%
An alternative  discrete model defined over cells was introduced in  \cite{MS11}.  
 The average discrete derivatives were defined, e.g.,  as
\begin{equation}
\label{avDiscrDer3D}
\begin{split}
 \overline{\DD}_{e_1} v_{\ell}   = \frac 1 4\, \Big \{ \DD_{e_1} v_{\ell}    +\DD_{e_1} v_{\ell+ e_2} +\DD_{e_1} v_{\ell+ e_3} +\DD_{e_1} v_{\ell+ e_2+e_3}       \Big \}.
\end{split}
\end{equation}
This leads to a discrete gradient $\overline {\nabla} y$ in analogy to \eqref{DT}; see  \cite{MS11} for details.   The corresponding   \emph{{cell atomistic Cauchy--Born energy}} is then defined by
\begin{equation*}
\begin{split}
 {\Phi}^ {a, CB} (y) &:=  {\varepsilon  ^ 3}    \sum _{\ell \in \Lat } \ \,       
 \sum _{\eta \in R}\  \phi _\eta \, (\overline {\nabla} y\,  \eta )
 =  {\varepsilon  ^ 3}   \sum _{\ell \in \Lat } \,     W_{CB} ( \overline {\nabla}y    ).
 \end{split}
\end{equation*}

The corresponding  \emph{{cell atomistic Cauchy--Born problem}}  is 
\begin{equation*}
\begin{split}
  & \text{find a local minimizer $y^{a, CB}$  in $\mathcal{X} $ of : }\\
 &\ \  {\Phi}^ {a, CB} (y^{a, CB}) - \< f , y^{a, CB} \> _\varepsilon   .
\end{split}
\end{equation*}
This atomistic model is  consistent as well, in the sense that 
\be
 \< D   {\Phi}^ {a, CB}  (y_F), v\>=  0, \quad y_F (x) = \F x\, ,
  \ee
for all   $v\in \Vh.$   
As before, this is implied by 
\be \begin{split}
\< D   {\Phi}^ {a, CB}  (y_F), v\>=&
  {\varepsilon  ^ 3}    \sum _{\ell \in \Lat } \     \sum _{\eta \in R}\  \phi _\eta \, (\overline  {\nabla} y_F \eta )\cdot \overline  {\nabla} v  \,\eta\\
=&\sum _{\eta \in R}\ 
\ph'(F\,\eta)\cdot
\sum _{\ell \in \Lat } \  {\varepsilon  ^ 3}    \,    \overline  {\nabla} v \, \eta\\
=&\sum _{\eta \in R}\ 
\ph'(F\,\eta)\cdot
\sum _{\ell \in \Lat } \ \,         \int _{K_\ell}    {\nabla} v\,  \eta\,  dx\\
=&\sum _{\eta \in R}\ 
\ph'(F\,\eta)\cdot
 \int _\Omega    {\nabla} v\,  \eta \, dx= 0\, . 
 \end{split}\ee
%
It was shown in  \cite{MS11} that this model is both energy- and variationally  consistent to second order in $\varepsilon,$ approximating the exact atomistic model as well as the  continuum Cauchy-Born model.

\section{Bond volumes and long range differences}
To  construct methods that couple the atomistic and continuum descriptions  we need to relate long range differences and derivatives of functions defined over  \emph{bond volumes}. 
 To fix ideas, let $\eta\in R,$ and define the \emph{bond} as the line segment  $b_{\ell} =\{x\in\RR^3\colon {x=x_{\ell+t\eta}},\; 0<t<1\}$ with endpoints $x_\ell$ and $x_{\ell+\eta}$. The set of all bonds  $\mathcal{B}_\eta$ consists of all $b=b_\ell $ for $\ell\in \Lat$ (but for $\eta$ fixed).
%
%
For  given $\ell$ and $\eta\in\mathbb{Z}^3$ with $\eta_1\eta_2\eta_3\ne 0,$ the corresponding  \emph{bond volume} $B_{\ell, \, \eta} $  is the interior of the  rectangular parallelepiped  with  edges parallel to the standard basis vectors $e_i$  and main diagonal  $b_{\ell}$, see Fig. 3.
Next we shall establish a connection between long range differences and piecewise linear functions
defined over  type A  decompositions of bond volumes $B_{\ell, \, \eta} $ into tetrahedra, which is defined in analogy to type A  decompositions of cells $K_\ell.$ To this end 
let  $B_{\ell, \, \eta} (T) $ a \emph{type A decomposition} of the 
bond volume  $B_{\ell, \, \eta} $  into six tetrahedra, i.e.,  the decomposition were   the  diagonals 
$(x_{\ell} , x _{\ell+e_1 \eta_1 +e_3 \eta_3})$
 and $(x_{\ell+e_2\eta_2} , x _{\ell+\eta})$ are edges of the resulting tetrahedra, see Fig. 3.
\begin{figure}
    \begin{center}
      \includegraphics[width=0.35\textwidth]{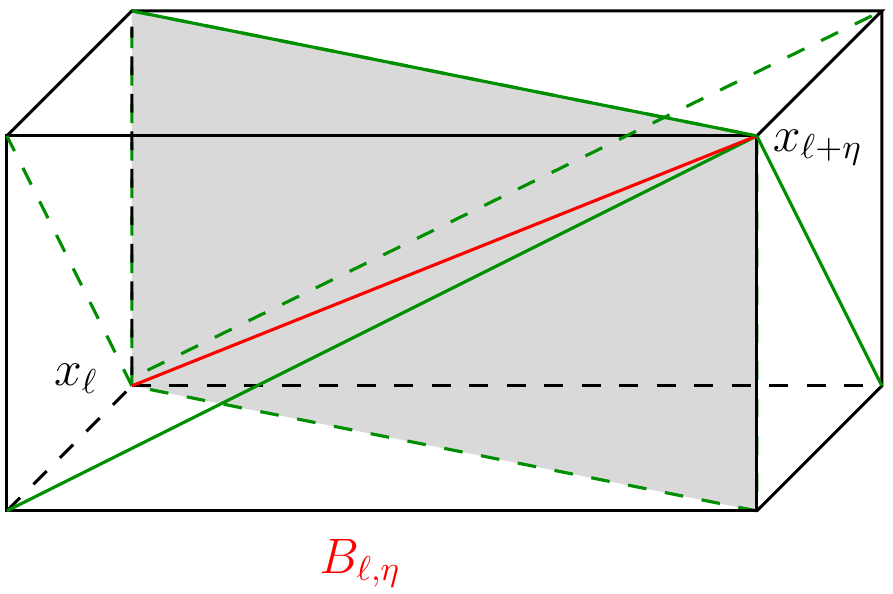}
      \caption{ A \emph{bond volume} $B_{\ell, \, \eta} $ and its  \emph{type A decomposition}    into six tetrahedra.}
    \label{Fig:3}
   \end{center}
\end{figure}

The following  lemma plays a central role in our work.  

\begin{lemma}\label{bv_lemma} Let $v $ be a \emph{piecewise linear and continuous} function on
a type  A  decomposition of the bond volume $B_{\ell, \, \eta} $ into tetrahedra.   Then
\be 
\varepsilon ^3 \, \overline D_\eta v_\ell = \frac {1} { | \eta_1\, \eta_2\, \eta_3 |} \, \int _ {B_{\ell, \, \eta}}\, \nabla v (x)  \eta\, dx\, .
\ee
\end{lemma}
\begin{proof} {We present the proof for $\eta_i>0, $ $i=1,2,3$. The other cases are similar.}
We have, 
\be\label{BelletaC}
\begin{split}
\frac {1} { \eta_1\, \eta_2\, \eta_3} \, \int _ {B_{\ell, \, \eta}}\, \nabla v (x)  \eta \, dx&
= \frac {1} { \eta_1\, \eta_2\, \eta_3} \int _ {\partial B_{\ell, \, \eta}}  v \,\nu \cdot \eta   \, ds \, \\
= \frac {1} { \eta_1\, \eta_2\, \eta_3} 
&\sum _{i=1}^3  \Big \{ \int _ {   \partial B_{\ell, \, \eta}   (-e_i)}\,  (-\eta_i) v  \, ds + \int _   {\partial B_{\ell, \, \eta}   (e_i) }\,   \eta _i \,  v  \, ds \, \Big \}\, , 
\end{split}
\ee
where  $\partial B_{\ell, \, \eta}   (e_i)$ is the face of $B_{\ell, \, \eta} $ with outward unit normal  $e_i.$
Since $v$ is linear in each tetrahedron of the decomposition of $B_{\ell, \, \eta}, $ it will be 
linear in each  of the two triangles comprising the face   $\partial B_{\ell, \, \eta}   (\eta_i).$  
Therefore, if $\tau$ is such a triangle, the  integral of $v$ over $\tau$ can be found explicitly:
\be
  \int _ {\tau}\,   \eta _i \,  v  \, ds= \frac {|\tau |} 3 \sum _{j=1} ^3 \ \eta _i \,  v(z_j) ,  
\ee
where $z_i$  are the vertices of  $\tau$. 
Since $\tau$ is one of the two triangles of   $\partial B_{\ell, \, \eta}   (\eta_i)$,
${|\tau |}   \ \eta _i  = \frac {\varepsilon ^2} 2 \  \eta_1\, \eta_2\, \eta_3$. 
Hence, 
\be
\frac {1} { \eta_1\, \eta_2\, \eta_3}
 \int _   {\partial B_{\ell, \, \eta}   (e_i) }\,   \eta _i \,  v  \, ds \, 
 =  \frac {\varepsilon ^2 } 6 \sum _{j=1} ^2 \ \big  \{  v(z_j) +  2\, v(\tilde z_j) \big  \}  , 
\ee
where $\tilde z_j$ are the vertices shared by two triangles of  $\partial B_{\ell, \, \eta}   (e_i)$
and $z_j$ the vertices belonging to only one triangle of  $\partial B_{\ell, \, \eta}   (e_i)$.
 
We substitute the above formula into \eqref{BelletaC} and group together all
terms involving each vertex. For each of the vertices other than $ x_{\ell}$ or $ x_{\ell +\eta}$, there
are two possibilities: 
\begin{enumerate}
\item[(i)] It is a  shared vertex   in one face with outward normal $e_i$ and it is a single vertex 
        in \emph{two} faces  with normal $-e_i$.
\item[(ii)] It is a  shared vertex  in one face with normal $-e_i$ and a single vertex in \emph{two} faces 
       with normal $e_i$. 
\end{enumerate}
Also, terms involving a vertex of $ \partial B_{\ell, \, \eta}   (e_i)$ appear with coefficient $1$, while 
terms involving a vertex of $ \partial B_{\ell, \, \eta}   (-e_i)$ appear with coefficient $-1$ in
\eqref{BelletaC}. 
Therefore the contribution of these vertices to the sum in \eqref{BelletaC}  is zero. 

Finally, we notice  that
$x_{\ell +\eta}$ is a shared vertex at each $\partial B_{\ell, \, \eta}   (e_i)$,
while $x_{\ell }$ is a shared vertex at each $\partial B_{\ell, \, -\eta}   (-e_i)$, for all $ i=1, 2, 3$.
It follows that
\be
\begin{split}
\frac {1} { \eta_1\, \eta_2\, \eta_3 } \, \int _ {B_{\ell, \, \eta}}\, \nabla v (x) \cdot \eta \, dx = \varepsilon ^2 \, \big ( v _{\ell+ \eta }     - v_  {\ell } \big ) , 
\end{split}
\ee
and the proof is complete.  
\end{proof}

\begin{figure}
    \begin{center}
      \includegraphics[width=0.5\textwidth]{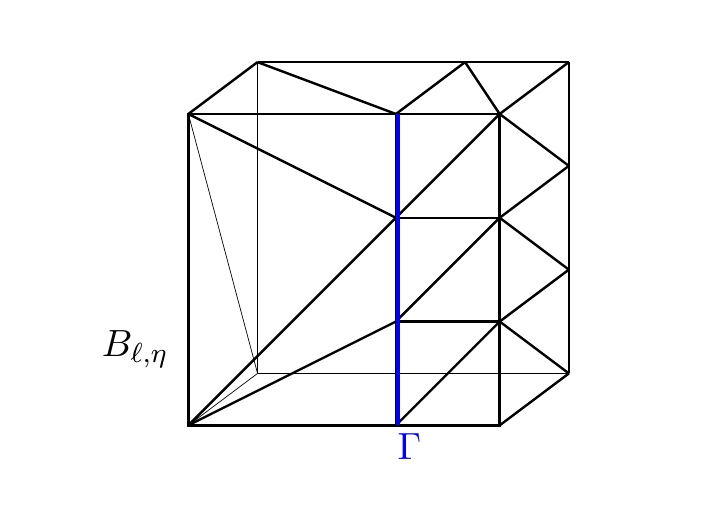}
      \caption{A possible  decomposition $\T (B_{\ell, \, \eta}) $ of $B_{\ell, \, \eta}.$}
    \label{Fig:2}
   \end{center}
\end{figure}

\section{A coupling method based on   bond volumes }  

In this section we construct methods based on \emph{bond volumes}. 
Let the \emph{atomistic region} $\Om_a$ and the \emph{A-CB region} $\Om_*$ each be  the interior of the closure of a union of lattice 
tetrahedra
$T\in \T _T $ and connected, and suppose 
$$\Om=\overline{\Om}_a\cup\overline{\Om}_*, \quad \GG=\overline{\Om}_a\cap\overline{\Om}_* .$$
Here $\GG$ is the interface.
To avoid technicalities that may arise due to the fact that we work with  periodic functions
over $\Om$,  we assume throughout  that $\overline{ \Om}_a$ is subset of the {interior} of $\Om$ with sufficient distance from $\partial \Omega.$   Let $y_\ell$ be the deformed position of $x_\ell\in \Omega_{\text{discr}}$. 

Fix $\eta\in R, $ with $\eta_1\eta_2\eta_3\ne 0.$ The cases of degenerate $\eta$ can be treated with 
two and one dimensional techniques. We shall construct an energy based  coupling method whose   design relies on an appropriate handling of  bond volumes $B_{\ell, \, \eta}$.  We consider three cases depending on the
location of each bond volume $B_{\ell, \, \eta} $:
\begin{enumerate} 
\item[(a)] The closure of the bond volume is contained in the atomistic region: 
        $\overline B_{\ell, \, \eta} \subset\Om_a$.
\item[(b)] The bond volume is contained in the  region  $\Om_*$:  $B_{\ell, \, \eta} \subset\Om_*$.
\item[(c)] We denote by $ B_\Gamma$ the set of bond volumes that do not satisfy (a) or (b). 
         In fact,  $B_{\ell, \, \eta}\in  B_\Gamma$ if the bond volume intersects  the interface:      
         $B_{\ell, \, \eta} \cap\Gamma\neq\emptyset$ or if $  B_{\ell, \, \eta} \subset\Om_a$ and  
         $\overline B_{\ell, \, \eta} \cap\Gamma\neq\emptyset$.
\end{enumerate} 
If a bond volume intersects $\partial \Omega,$ then it is supposed to belong to $\Om_*$ by periodic extension.  
For a fixed $\eta,$ the contribution to the energy corresponding to the atomistic region (case (a)) is
 \be\label{energy:a}
  E_{\Om_a, \eta }^a\{y\} = \varepsilon ^3 \sum_{\substack{\ell\in \Lat \\\overline B_{\ell, \, \eta} \subset\Om_a}}\ph(\overline D_\eta y_\ell )\, . 
   \ee
The contribution to the energy from the A-CB region (case (b)) is (cf. \eqref{ACB-tetrah}), 
\be  \label {overline_y}   E_{\Om_*, \eta }^{a,cb}\{y\} =
\frac {\varepsilon  ^ 3} 6    \,   \sum _{\ell \in \Lat, \ T \in K_\ell(T), \ T \subset \Om_* }    \  \phi _\eta \, (\widetilde {\nabla} y\,  \eta )=\int_{\Om_*}\ph({\nabla}  \overline y(x)\eta)dx\, , 
      \ee
$\overline y$ being the interpolant of $\{y_\ell\}$ in $\Wh(\Om_*)$, see below \eqref{Vet}.

For each bond volume intersecting  $\Gamma$ we denote by $y^{\ell, \eta}$ a piecewise polynomial  function on 
$B_{\ell, \, \eta} $ 
satisfying 
 \begin{enumerate} 
\item[i)] $y^{\ell, \eta}\in C (\overline {B}_{\ell, \, \eta} )$.
\item[ii)] Let $\T (B_{\ell, \, \eta}) $ be  a decomposition 
of  $ B_{\ell, \, \eta}$ with the following properties: a) if  {$T \in \T (B_{\ell, \, \eta})$ and $T \subset \Om_*$} then 
{$T$}  is a tetrahedron resulting from a type A decomposition of an atomistic cell $K\subset\Om_*$. 
 b) If   {$T \in \T (B_{\ell, \, \eta})$ and $T \subset \Om_a$, then $T$} is a lattice tetrahedron. 
 %
\item[iii)] {In  case ii.b)   above,  if $T$ has a face on   $\partial \big ( B_{\ell, \, \eta}\cap \Om_a \big ) \backslash \Gamma$, then it is part of a conforming decomposition that is compatible with 
decompositions of other bond volumes sharing a face with $B_{\ell,\eta}$. If such an attached bond volume is included in $\Om_a$, then it is assumed to be  type-A decomposed into tetrahedra.}
   \item[iv)] For {$T \in \T (B_{\ell, \, \eta}),$} $y^{\ell, \eta} \in \mathbb{P} _1 (T)$ and it  interpolates 
                  $\{y_\ell\}$ at the vertices of {$T$}.
\end{enumerate}

Then the energy due to bond volumes intersecting  the interface is defined as
\begin{equation}\label{energy:GG}
 E_{\GG, \eta }^{}\{y\}= \sum_{
    \substack{\ell \in \Lat \\ B_{\ell, \, \eta}\in  B_\Gamma} }\frac 1 {|\eta_ 1\, \eta_ 2\,   \eta_ 3|} \int_{B_{\ell, \, \eta } }\chi_{{}_{\Om_a}}\ph(\nabla y^{\ell, \eta}\eta)\,  dx \, .
   \end{equation}
  \begin{remark} Notice that the energy that corresponds to the bond volume $B_{\ell, \, \eta }\in B_{\Gamma}$ would be 
  \begin{equation}
 \frac 1 {|\eta_ 1\, \eta_ 2\,   \eta_ 3|} \int_{B_{\ell, \, \eta } }\ph(\nabla y^{\ell, \eta}\eta)\,  dx \, .
   \end{equation}
The part of this energy corresponding to $  B_{\ell, \, \eta }\cap \Om_*$ has been already taken into account in  $E_{\Om_*, \eta }^{a,cb}\{y\}$ and hence it is not 
included in the definition of $E_{\GG, \eta }^{}\{y\}\, .$  
 \end{remark}
  \begin{remark} The choice of the decomposition $\T (B_{\ell, \, \eta}) $ and of the associated    piecewise polynomial  function $y^{\ell, \eta}$ is somewhat  flexible; see {\cite{MMR12} for a more detailed discussion.} It might even allow
  vertices that are not lattice points.  The only essential requirement 
  is 
that each function $v ^ {[m]}   $ defined through $\T (B_{\ell, \, \eta}) $  in the proof of Proposition \ref{bv_gffree1}  below, should satisfy $v ^ {[m]}  \in H^1 (\Omega)\, .$ Depending on the complexity of the
interface $\Gamma$ one can construct  such decompositions more or less efficiently.  In many cases this can simplify  the computation of the associated energy $ E_{\GG, \eta }^{}\{y\}$. 
See for example, Figure 4 for such a choice of decomposition.   

\end{remark}

We then define the total energy as follows
\be  \label{E_bv}  \mathcal{E} _{bv}\{y\} = \sum _{\eta \in R}  \mathcal{E} _{\, \eta}\{y\}   \ee
where 
 \be
  \mathcal{E} _{\, \eta}\{y\}=  E_{\Om_a, \eta }^a\{y\} +  E_{\Om_*, \eta }^{a,cb}\{y\} + E_{\GG, \eta }^{}\{y\}\, .
   \ee

  
\subsection{Consistency}
The  energy \eqref{E_bv} based on bond volumes is ghost-force free, as  we  prove in the following proposition.
\begin{proposition}\label{bv_gffree1} The energy \eqref{E_bv} is free of ghost forces, in the sense that 
\be
 \< D \mathcal{E} _{bv} (y_F), v\>=  0, \quad y_F (x) = \F x\, , \ee
for all { $ v\in \V $.
   }
\end{proposition}

\noindent
To show this proposition we shall need some more notation. 
 First we fix $\eta$ and  consider decompositions into  bond volumes which cover  $\mathbb{R}^3$:
\begin{equation}
\begin{split}
&\mathcal{S} ^m_{B_\eta}  :=\Big  \{ B_{\ell, \, \eta}  \, : \  
\text{(i)}\  B_{\ell, \, \eta}  \cap B_{j, \, \eta}=\emptyset, \text { if } \ell \ne j,  \quad \text{(ii)} \ \mathbb{R}^2  =
\overline{\cup  B_{\ell, \, \eta} }   \  \Big \}\, , \quad m=1, \dots, |\eta_ 1\, \eta_ 2\,   \eta_ 3|. 
\end{split}
\end{equation}
$\mathcal{S} ^m_{B_\eta} $ will be used for counting purposes in the proof;  the associated functions 
introduced below will be defined on $\Omega.$
The number of different such coverings is $|\eta_ 1\, \eta_ 2\,   \eta_ 3|\, ,$  hence the numbering $m=1, \dots, |\eta_ 1\, \eta_ 2\,   \eta_ 3|.$ Notice that bond volumes corresponding to different $m$ may overlap, but the elements of a single $\mathcal{S} ^m_{B_\eta} $ are  non-overlapping bond volumes.      

For a lattice function $\{v_\ell\}$  construct the functions ${\nabla}  \overline v$ and $ v^{\ell, \eta} $   in analogy  with $ {\nabla}  \overline y$ and $ y^{\ell, \eta}$ in the construction below  \eqref{overline_y}. 
Then for a fixed $\eta$ we have
\be \begin{split}
\< D\mathcal{E} _{\, \eta} (y_F), v\>=  
\ph'(F\,\eta)\cdot
\Big \{ &\varepsilon ^3 \sum_{\substack{\ell\in \Lat \\\overline B_{\ell, \, \eta} \subset\Om_a}}\,  \overline D_\eta v_\ell 
+\int_{\Om_*}{\nabla}  \overline v(x)\eta\, dx\\\
&+ \sum_{
    \substack{\ell \in \Lat \\ B_{\ell, \, \eta}\in  B_\Gamma } }\frac 1 {| \eta_1\, \eta_2\, \eta_3 |} \int_{B_{\ell, \, \eta } }\chi_{{}_{\Om_a}} \nabla v^{\ell, \eta} \eta\,  dx
\Big \}\, .
 \end{split}\ee
 The main idea in the proof of Proposition \ref{bv_gffree1} is to rewrite the expression within brackets above  in the following way: 
  \be \label{idea}\begin{split}
  \varepsilon ^3  \sum_{\substack{\ell\in \Lat \\\overline B_{\ell, \, \eta} \subset\Om_a}}\,  \overline D_\eta v_\ell 
&+\int_{\Om_*}\overline{\nabla}  v(x)\eta\, dx\
+ \sum_{
    \substack{\ell \in \Lat \\ B_{\ell, \, \eta}\in  B_\Gamma } }\frac 1 {| \eta_1\, \eta_2\, \eta_3 |} \int_{B_{\ell, \, \eta } }\chi_{{}_{\Om_a}} \nabla v^{\ell, \eta} \eta\,  dx\\
&=\frac 1 {| \eta_1\, \eta_2\, \eta_3 | } \sum _{m= 1} ^{| \eta_1\, \eta_2\, \eta_3 |} \, \int_{\Omega }  {\nabla}  v^ {[m]}(x)\eta\, dx,
 \end{split}\ee
 where $v^ {[m]}, $  $ m =1, \dots, | \eta_1\, \eta_2\, \eta_3 |$  are appropriate conforming functions in $H^1(\Omega) $, each  associated to a different  covering $\mathcal{S} ^m_{B_\eta} $
 consisting of bond volumes.
 The details are provided below. 
 
 \noindent
 \emph{Proof of Proposition \ref{bv_gffree1}.}  We use Lemma \ref{bv_lemma} to write 
 
 \be \begin{split}\label{prop2-1}
  \varepsilon ^3  \sum_{\substack{\ell\in \Lat \\\overline B_{\ell, \, \eta} \subset\Om_a}}\,  \overline D_\eta v_\ell 
= \frac 1 {| \eta_1\, \eta_2\, \eta_3 | } \sum_{\substack{\ell\in \Lat \\\overline B_{\ell, \, \eta} \subset\Om_a}}   \, \int _ {B_{\ell, \, \eta}}\, \nabla v ^ {[m]} (x)  \eta\, dx
 \end{split}\ee 
where $v ^ {[m]}  $   is a \emph{piecewise linear  continuous} function on
a type  A  decomposition of the bond volume $B_{\ell, \, \eta} $ into tetrahedra.  The superscript
$m$ indicates  the covering  $\mathcal{S} ^m_{B_\eta}$ to which  $B_{\ell, \, \eta}$ belongs. In fact,  $v ^ {[m]}  $ can be defined globally as
follows: For a given lattice function  $\{v_\ell\},$ a fixed  $m$ and a  covering  $\mathcal{S} ^m_{B_\eta}$  $v ^ {[m]}  $ is equal to
\begin{enumerate} 
\item[-] the  {piecewise linear} interpolant of $\{v_\ell\}$ on
a type  A  decomposition of the bond volume $B_{\ell, \, \eta} $ into tetrahedra  if  $\overline B_{\ell, \, \eta} \subset\Om_a,$
\item[-]  $v^{\ell, \eta},$ for  $B_{\ell, \, \eta}\cap\Gamma\neq\emptyset, $ where the piecewise polynomial $v^{\ell, \eta}$ on   $B_{\ell, \, \eta}$ is defined through (i--iv) below \eqref{overline_y}, 
\item[-]  the   piecewise linear function interpolating    $\{v_\ell\}$   {at  lattice tetrahedra   $T  \subset B_{\ell, \, \eta} \subset\Om_*.$}
\end{enumerate} 
It is clear by construction that each $v ^ {[m]}  \in H^1 (\Omega)\, .$ Further, each tetrahedron corresponds to exactly one atomistic cell $K   \subset\Om_* $ belonging to $|\eta_1 \eta_2 \eta_3|$ different 
bond volumes $B_{\ell, \, \eta}$,  each one belonging to a different covering $\mathcal{S} ^m_{B_\eta}.$ Thus for $T  \subset\Om_* $  we have  
\be\label{prop2-2}
\int_{T}{\nabla}  \overline v(x)\eta\, dx = \frac 1 {| \eta_1\, \eta_2\, \eta_3 | }   \sum _{m= 1} ^{| \eta_1\, \eta_2\, \eta_3 |}  \int _ {T\cap B_{\ell, \, \eta} \in \mathcal{S} ^m_{B_\eta}} \nabla  v^{[m]}(x)\eta\, dx\, .
\ee
Therefore, 
\be \begin{split}
 \int_{\Om_*} {\nabla}  \overline v(x)\eta\, dx= \frac 1 {| \eta_1\, \eta_2\, \eta_3 | }   \sum _{m= 1} ^{| \eta_1\, \eta_2\, \eta_3 |}  \int _ {\Om_*} \nabla  v^{[m]}(x)\eta\, dx \, . \end{split}\ee
 By construction of  $v^{\ell, \eta} $ and $v ^ {[m]}  $ we have
 \be\label{vmGamma}
 \sum_{
    \substack{\ell \in \Lat \\B_{\ell, \, \eta}\in  B_\Gamma } }\frac 1 {| \eta_1\, \eta_2\, \eta_3 |} \int_{B_{\ell, \, \eta } }\chi_{{}_{\Om_a}} \nabla v^{\ell, \eta} \eta\,  dx
 =\frac 1 {| \eta_1\, \eta_2\, \eta_3 | }   \sum _{m= 1} ^{| \eta_1\, \eta_2\, \eta_3 |}  \  \sum _{\substack{B_{\ell, \, \eta} \in \mathcal{S} ^m_{B_\eta} \\B_{\ell, \, \eta}\in  B_\Gamma } }\ 
 \int_{B_{\ell, \, \eta } }\chi_{{}_{\Om_a}}
 \nabla  v^{[m]}(x)\eta\, dx
 \ee
Thus rewriting \eqref{prop2-1} as 
\be \begin{split}\label{prop2-3}
  \varepsilon ^3  \sum_{\substack{\ell\in \Lat \\ \overline B_{\ell, \, \eta} \subset\Om_a}}\,  \overline D_\eta v_\ell 
=\frac 1 {| \eta_1\, \eta_2\, \eta_3 | }   \sum _{m= 1} ^{| \eta_1\, \eta_2\, \eta_3 |}  \  \sum _{\substack{B_{\ell, \, \eta} \in \mathcal{S} ^m_{B_\eta} \\ \overline B_{\ell, \, \eta} \subset\Om_a } }\ 
 \int_{B_{\ell, \, \eta } } 
 \nabla  v^{[m]}(x)\eta\, dx, 
 \end{split}\ee 
we finally obtain
\be
\begin{split}\label{vmGamma-2}
\varepsilon ^3  \sum_{\substack{\ell\in \Lat \\ \overline B_{\ell, \, \eta} \subset\Om_a}}\,  \overline D_\eta v_\ell  &+ \sum_{
    \substack{\ell \in \Lat \\ B_{\ell, \, \eta}\in  B_\Gamma } }\frac 1 {| \eta_1\, \eta_2\, \eta_3 |} \int_{B_{\ell, \, \eta } }\chi_{{}_{\Om_a}} \nabla v^{\ell, \eta} \eta\,  dx\\
& =\frac 1 {| \eta_1\, \eta_2\, \eta_3 | }   \sum _{m= 1} ^{| \eta_1\, \eta_2\, \eta_3 |}  \  \sum _{\substack{B_{\ell, \, \eta} \in \mathcal{S} ^m_{B_\eta} \\ B_{\ell, \, \eta}\in  B_\Gamma   \  \text{ or } \ \overline B_{\ell, \, \eta} \subset\Om_a }  }\ 
 \int_{B_{\ell, \, \eta } }\chi_{{}_{\Om_a}}
 \nabla  v^{[m]}(x)\eta\, dx\\
& =\frac 1 {| \eta_1\, \eta_2\, \eta_3 | }   \sum _{m= 1} ^{| \eta_1\, \eta_2\, \eta_3 |}  \    \int_{ \Om_a}
 \nabla  v^{[m]}(x)\eta\, dx \, . 
 \end{split}
 \ee
Hence \eqref{idea} follows in view of \eqref{prop2-2}. Therefore the proof of proposition is complete in view of the Gauss-Green theorem.
\qed 

\section{The discontinuous   bond volume based coupling method}  

In this section we show that is is possible to modify energies to allow underlying functions
which might be discontinuous at the interface.
This allows greater flexibility on the construction of the underlying meshes and thus the computation
of the energy at the interface might become simpler. To retain consistency the interfacial 
energies should include terms accounting for the possible 
discontinuity of the underlying functions. 
There are  many alternatives, such as the possibility of adding extra stabilization terms, compare to 
\cite{BallMoral09}. The purpose of this paper is however to present the general
framework and we will not insist on the various modifications and extensions of the methods developed herein.

Let $\Om$, $\Om_a,$ $\Om_*$  and $\GG$  be   as in the previous section.   
 Further,   we distinguish the same  cases  a), b) and  c) regarding the 
location of each bond volume $B_{\ell, \, \eta} .$
The corresponding energies are still defined by 
 \be E_{\Om_a, \eta }^a\{y\}= \varepsilon ^3 \sum_{\substack{\ell\in \Lat \\\overline B_{\ell, \, \eta} \subset\Om_a}}\ph(\overline D_\eta y_\ell ) ,
   \ee
and
\be  \label {overline_y:DG}   E_{\Om_*, \eta }^{a,cb}\{y\}
=
\frac {\varepsilon  ^ 3} 6    \,   \sum _{\ell \in \Lat, \ T \in K_\ell(T), \ T \subset \Om_* }    \  \phi _\eta \, (\widetilde {\nabla} y\,  \eta )=\int_{\Om_*}\ph({\nabla}  \overline y(x)\eta)dx\, , 
      \ee
$\overline y$ being the piecewise linear function at the lattice tetrahedra  interpolating $\{y_\ell\}$.     
The main difference to the previous construction in Section 3
is the choice of  $y^{\ell, \eta}$ and the corresponding energies for each bond volume intersecting  the interface. 
In fact we let    
\begin{enumerate} 
\item[i)] $y^{\ell, \eta}\in C (\overline {B_{\ell, \, \eta}}\  \backslash \GG ) .$
\item[ii)] Further, let $\T (B_{\ell, \, \eta}) $ be  a decomposition 
of  $ B_{\ell, \, \eta}$ with the properties a) if  ${T} \in \T (B_{\ell, \, \eta})$ and 
$ {T} \subset \Om_*$ then $ {T}$  is an atomistic tetrahedron resulting from a type A decomposition of an atomistic cell. 
 b) If   $ {T} \in \T (B_{\ell, \, \eta})$ and ${T} \subset \Om_a$ then 
 $ {T}$ is a  {lattice tetrahedron}. 
\item[iii)] In the  case ii.b)   above  if $ {T}$ has a face on   $\partial \big ( B_{\ell, \, \eta}\cap \Om_a \big ) \backslash \Gamma$ then it will allow for a compatible conforming
decomposition with respect to attached  bond volumes. In that case if the attached bond volume 
is included in $\Om_a$ it is assumed to be  type-A decomposed into tetrahedra. 
 
   \item[iv)] For $ {T} \in \T (B_{\ell, \, \eta}),$ $y^{\ell, \eta} \in \mathbb{P} _1 ( {T}), $ interpolating $\{y_\ell \}\,$ at the vertices of $ {T}$.
\end{enumerate} 
We have kept the same properties, but we allow discontinuous matching across the interface $\GG. $ This provides greater flexibility on the construction of $y^{\ell, \eta}$ since  {it allows the presence of } arbitrary
hanging nodes on the interface of the two regions.  

We then define  the energy due to bond volumes intersecting  the interface  as
\begin{equation}
 E_{\GG, \eta }^{D}\{y\}   = \sum_{
    \substack{\ell \in \Lat \\ B_{\ell, \, \eta}\in  B_\Gamma} }\frac 1 {|\eta_ 1 \eta_ 2 \eta_3|} \Big [ \int_{B_{\ell, \, \eta } }\chi_{{}_{\Om_a}}\ph(\nabla y^{\ell, \eta}\eta)\,  dx  
    - \int _{B_{\ell, \, \eta } \cap \GG } \ph '(\Av{\nabla y^{\ell, \eta}\eta})\,  {\cdot} \J {y^{\ell, \eta}\eta }\,  dS \, \Big ] .
   \end{equation}
 Here,  {$\J {w\eta}$}, $\Av w$ denote the jump and the average of a possibly discontinuous function 
 on the interface
 \be\label{Jw}
  {\J {w\eta}  : = (\nu _{\Omega _a}\cdot\eta)\, w^-\,  + (\nu _{\Omega _*}\cdot\eta)\, w^+ } , \qquad 
  \Av w  : = \frac 1 2 \{  w^-\,  + w^+ \} \, ,
 \ee 
 where $ w^-$ and $ w^+$ being the limits taken from $\Omega _a$ and
$\Omega _*$ respectively, and  $\nu _{\Omega _a}$,  $ \nu _{\Omega _*}$ the corresponding 
exterior normal unit vectors, with  $\nu _{\Omega _a}=- \nu _{\Omega _*}$ on $\GG$. 

A key observation here is that $E_{\GG, \eta }^{D}\{y\} $ \emph{does not induce inconsistencies} on the energy level. In fact, it is obvious that if  $y^{\ell, \eta} \in C (\overline {B_{\ell, \, \eta}}  )$
as in the previous section, then  
\be\label{Jw}
  E_{\GG, \eta }^{D}\{y\}   = E_{\GG, \eta }^{ }\{y\} \, , 
 \ee   
since the extra term on the interface vanishes. 
Then, as in the previous section, we define the total energy as follows:
\be  \label{E_bv:DG}  \mathcal{E} _{bv} ^D \{y\} = \sum _{\eta \in R}  \mathcal{E} ^D _{\, \eta}\{y\} ,
\ee
where 
 \be
  \mathcal{E} ^D_{\, \eta}\{y\}=  E_{\Om_a, \eta }^a\{y\} +  E_{\Om_*, \eta }^{a,cb}\{y\} + E _{\GG, \eta }^{D}\{y\}\, .
   \ee
Despite the fact that we allow discontinuities, the energy $\mathcal{E} _{bv} ^D$   is still  ghost-force free:
\begin{proposition}\label{bv_gffree1:DG} The energy \eqref{E_bv:DG} is free of ghost forces, in the sense that 
\be
 \< D \mathcal{E} _{bv} ^D (y_F), v\>=  0, \quad y_F (x) = \F x , 
\ee
for all  { $ v\in \V $.
   }
\end{proposition}

\begin{proof}
The structure of the proof is the same to that of Proposition \ref{bv_gffree1}, hence  we present 
in detail only the main differences. 
We still need the coverings  $\mathcal{S} ^m_{B_\eta} $ and recall that their elements   
{define} a decomposition of
  non-overlapping bond volumes.      
 {As in the} proof of Proposition \ref{bv_gffree1}
for a given lattice function $\{v_\ell\}$  we define  the functions ${\nabla}  \overline v$ and $ v^{\ell, \eta} $   in analogy  with $ {\nabla}  \overline y$ and $ y^{\ell, \eta}, $ cf.,  \eqref{overline_y}. 
Then,  we have
\be \begin{split}
\< D\mathcal{E} ^D  _{\, \eta} (y_F), v\>=  
\ph'(F\,\eta)\cdot
\Big \{ &\varepsilon ^3 \sum_{\substack{\ell\in \Lat \\\overline B_{\ell, \, \eta} \subset\Om_a}}\,  \overline D_\eta v_\ell 
+\int_{\Om_*}{\nabla}  \overline v(x)\eta\, dx\\\
&+ \frac 1 {| \eta_1\, \eta_2\, \eta_3 |} \sum_{
    \substack{\ell \in \Lat \\ B_{\ell, \, \eta}\in  B_\Gamma } }\int_{B_{\ell, \, \eta } }\chi_{{}_{\Om_a}} \nabla v^{\ell, \eta} \eta\,  dx -  \int _{B_{\ell, \, \eta } \cap \GG }   \J {v^{\ell, \eta}\eta }\,  dS \,  
\Big \}\, .
 \end{split}\ee
 Indeed, to show this it suffices to evaluate 
\be
\< D \big [   \ph '(\Av{\nabla w\eta} ) {\cdot}  \J {w \eta }  \, \big ], v\>=
   {\Bigl(}
  \ph ''(\Av{\nabla w \eta} )\, \J {w \eta}  {\Bigl) \cdot}
   \Av{\nabla v\eta}   
+  \ph '(\Av{\nabla w\eta } ) {\cdot} \J {v \eta} ,
  \ee
which  is equal to  $  {\ph '(F\eta )\cdot \J {v \eta } }$  for $w= y_F$.


In parallel to the proof of Proposition \ref{bv_gffree1} we shall prove
  \be \label{idea:DG}\begin{split}
  \varepsilon ^3  \!\! \! \sum_{\substack{\ell\in \Lat \\\overline B_{\ell, \, \eta} \subset\Om_a}}\,  \overline D_\eta v_\ell 
&+\int_{\Om_*}\overline{\nabla}  v(x)\eta\, dx\
+ \!\!\!\! \sum_{
    \substack{\ell \in \Lat \\ B_{\ell, \, \eta}\in  B_\Gamma } }\frac 1 {| \eta_1\, \eta_2\, \eta_3 |} \, \Big [\int_{B_{\ell, \, \eta } }\chi_{{}_{\Om_a}} \nabla v^{\ell, \eta} \eta\,  dx-  \int _{B_{\ell, \, \eta } \cap \GG }   \J {v^{\ell, \eta}\eta }\,  dS \,\Big ]  \\
&=\frac 1 {| \eta_1\, \eta_2\, \eta_3 | } \sum _{m= 1} ^{| \eta_1\, \eta_2\, \eta_3 |} \,\ \Big [ \int_{\Omega \backslash \GG }  {\nabla}  v^ {[m]}(x)\eta\, dx   -  \int _{ \GG }   \J {v^ {[m]} \eta }\,  dS \, \Big ]\\\
 \end{split}\ee
 where $v^ {[m]}, $  $ m =1, \dots, | \eta_1\, \eta_2\, \eta_3 |$,  are appropriate  functions in $H^1(\Omega\backslash \GG)\cap C (\overline \Omega \backslash \GG) ,$ possibly discontinuous at $\GG;$ each $v^ {[m]} $ is  associated to a different  covering $\mathcal{S} ^m_{B_\eta} $
 consisting of bond volumes. Relation \eqref{idea:DG}  then implies  $\< D\mathcal{E} ^D  _{\, \eta} (y_F), v\>=  0$ since by the Gauss Green theorem, 
 $$
 \int_{\Omega \backslash \GG }  {\nabla}  v^ {[m]}(x)\eta\, dx   =  \int _{ \GG }   \J {v^ {[m]} \eta }\,  dS\, .
 $$
 It remains therefore to establish \eqref{idea:DG}. To this end, we proceed exactly 
 as in the proof of Proposition \ref{bv_gffree1}. In particular, 
 for a given lattice function  $\{v_\ell\}$, a fixed  $m$ and a  covering  $\mathcal{S} ^m_{B_\eta}$ define $v ^ {[m]}  $ as
\begin{enumerate} 
\item[-] the  {piecewise linear} interpolant of $\{v_\ell\}$ on
a type  A  decomposition of the bond volume $B_{\ell, \, \eta} $ into tetrahedra  if  $\overline B_{\ell, \, \eta} \subset\Om_a,$
\item[-]  $v^{\ell, \eta},$ for  $B_{\ell, \, \eta}\cap\Gamma\neq\emptyset, $ where the piecewise polynomial on $B_{\ell, \, \eta},$  $v^{\ell, \eta}$  is possibly discontinuous on  $B_{\ell, \, \eta}\cap\Gamma,$   and  is defined through (i--iv) above, 
\item[-]  the piecewise linear function at the lattice tetrahedra  interpolating    $\{v_\ell\}$, if $T$  is an atomistic tetrahedron  such that   $ {T} \subset B_{\ell, \, \eta} \subset\Om_*$.
\end{enumerate} 
Now,  by construction  $v ^ {[m]}  \in H^1(\Omega\backslash \GG)\cap C (\overline \Omega \backslash \GG)$, and is possibly discontinuous at $\GG$.
The rest of the proof is identical to the one of  Proposition \ref{bv_gffree1} with the exception that 
\eqref{vmGamma} should be replaced by  
 \be
 \label{vmGamma:DG}
 \begin{split}
 \sum_{
    \substack{\ell \in \Lat \\B_{\ell, \, \eta}\in  B_\Gamma } }\frac 1 {| \eta_1\, \eta_2\, \eta_3 |} &\int_{B_{\ell, \, \eta } }\chi_{{}_{\Om_a}} \nabla v^{\ell, \eta} \eta\,  dx  -  \int _{B_{\ell, \, \eta } \cap \GG }   \J {v^{\ell, \eta}\eta }\,  dS\\
& =\frac 1 {| \eta_1\, \eta_2\, \eta_3 | }   \sum _{m= 1} ^{| \eta_1\, \eta_2\, \eta_3 |}  \  \sum _{\substack{B_{\ell, \, \eta} \in \mathcal{S} ^m_{B_\eta} \\B_{\ell, \, \eta}\in  B_\Gamma } }\ 
 \int_{B_{\ell, \, \eta } }\chi_{{}_{\Om_a}} 
 \nabla  v^{[m]}(x)\eta\, dx    -  \int _{B_{\ell, \, \eta } \cap \GG }   \J {v^{[m]}\eta }\,  dS,
  \end{split}
 \ee
 with \eqref{vmGamma-2} modified accordingly.  
\end{proof}

\section{High-order finite element coupling} 

In this section we shall see how the previous analysis can lead to energy-based methods which  employ high-order (even $hp$-) finite element approximations of the Cauchy-Born 
energy on the continuum region
while remaining ghost-force free. 
To this end let $\T _{ac}$ be a decomposition  of $\Omega$ into elements with the following properties: Let $\Om$, $\Om_a$ and $\Om_*$ as before with $\GG$   the interface.  The approximations will be based on 
decompositions of the continuum region  $\Om_*$ that are compatible
on  $\GG$ to  $\Wh (\Om_* ) .$ To this end, let 
\be \begin{split} \T_c (\Om_* ) \quad  &\text{be a  conforming decomposition  of } \ \Om_*   \ \text{ into tetrahedra with vertices lattice points,} \\
&\text{such that, if }\  T\in \T_c (\Om_* ), \overline T \cap \GG \neq\emptyset ,\quad \text{then}\quad
T\in \T_T (\Om_* )\, .
\end{split}
\ee
We consider the discrete space
\be \begin{split}
\Vac ( \Om_*) = \{ v \in   C ( \overline \Om_* ) :& \   v| _T \in \mathbb{P}_1 (T)\ \text{for }  \overline T \cap \GG \neq\emptyset  \text{ and }\\
&\  v| _T \in \mathbb{P}_k (T)\ \text{for all other  } T \in  \T_c (\Om_* ) \, \}\, . 
\end{split}\ee
This space can be extended to include the atomistic region as well by  
\be\begin{split}
\Vac = \{ v \in   C ( \overline \Om ) :  &\   v| _T \in \mathbb{P}_1 (T)\ \text{for }   T \in \T_T (\Om_a) \text{ and }\\
& v| _T \in \Vac ( \Om_*) \ \text{for all } T \in  \T_c (\Om_* ), \ \  v \ \text{periodic on } \Omega  \, \}\, . 
\end{split}\ee
For $v\in \Vac $ one can define the corresponding lattice  {function} $\{ v_\ell\}$, simply by interpolating. Conversely,  for given $\{ v_\ell\}$ one can find  $v\in \Vac $  that 
coincides with corresponding values of  $\{ v_\ell\}$   at the vertices. However, at the regions using high-order finite elements the other degrees of freedom should be defined with some care. 
In the following, we assume that we are given a function $y,$ $y (x) = \F  x  + v (x),$ $v \in \Vac $ and we shall define 
its atomistic/continuum energy. To this end, 
\be\label{E_ac}
\mathcal{E} _{\, h} = \sum_{\eta \in R}\  \mathcal{E} _{\, h, \eta}\, , 
\ee
where 
 \be
  \mathcal{E}_{\, h, \eta} \{y\}=  E_{\Om_a, \eta }^a\{y\}  +
   \int_{\Om_*}\ph( {\nabla}  y(x)\eta)dx + E_{\GG, \eta }^{}\{y\} \, .
   \ee
Here for an atomistic point  $x_\ell \in \Om_a, $  $y_\ell= y(x_\ell  ), $
and the local energies   $E_{\Om_a, \eta }^a\{y\}  ,$
  $  E_{\GG, \eta }^{}\{y\}   $ are defined as in Section 4, see
  \eqref{energy:a}, \eqref{energy:GG}.

The above method can designed to be of   arbitrary high order accuracy of the Cauchy-Born
energy at the continuum region $\Omega_* \, .$ Such methods are of importance since, by tuning the discretization parameters
(decomposition of  $\Omega_* \, $ and polynomial degrees) we have the possibility of matching the ideal accuracy at the continuum region
which is $O(\varepsilon ^2).$  
The energy $\mathcal{E} _{\, h}$ is ghost force free. 

\begin{proposition} The energy \eqref{E_ac} is free of ghost forces, in the sense that 
\be
 \< D \mathcal{E} _{\, h} (y_F), v\>=  0, \quad y_F (x) = \F x\, , \ee
for all   $v\in \Vac \, .$
\end{proposition}

\begin{proof} Since  $v\in \Vac$ we have 
\be \begin{split} \< D\mathcal{E} _{\, h, \eta } (y_F), v\>  
=
\ph'(F\,\eta)\cdot
\Big \{ &\varepsilon ^3 \sum_{\substack{\ell\in \Lat \\\overline B_{\ell, \, \eta} \subset\Om_a}}\,  \overline D_\eta v_\ell 
+\int_{\Om_*}{\nabla}   v(x)\eta\, dx\\\
&+ \sum_{
    \substack{\ell \in \Lat \\ B_{\ell, \, \eta}\in  B_\Gamma } }\frac 1 {| \eta_1\, \eta_2\, \eta_3 |} \int_{B_{\ell, \, \eta } }\chi_{{}_{\Om_a}} \nabla v^{\ell, \eta} \eta\,  dx
\Big \}\, .
 \end{split}\ee
 Exactly as in the proof of  of Proposition \ref{bv_gffree1} we may write the first and the third term
 of the the above sum as 
  \be \label{idea:HO}
  \begin{split}
  \varepsilon ^3  \sum_{\substack{\ell\in \Lat \\\overline B_{\ell, \, \eta} \subset\Om_a}}\,  \overline D_\eta v_\ell 
&
+ \sum_{
    \substack{\ell \in \Lat \\ B_{\ell, \, \eta}\in  B_\Gamma } }\frac 1 {| \eta_1\, \eta_2\, \eta_3 |} \int_{B_{\ell, \, \eta } }\chi_{{}_{\Om_a}} \nabla v^{\ell, \eta} \eta\,  dx\\
&=\frac 1 {| \eta_1\, \eta_2\, \eta_3 | } \sum _{m= 1} ^{| \eta_1\, \eta_2\, \eta_3 |} \, \int_{\Om_a} {\nabla}  v^ {[m]}(x)\eta\, dx\\\
 \end{split}\ee
 where $v^ {[m]}, $  $ m =1, \dots, | \eta_1\, \eta_2\, \eta_3 |$ are  the functions defined in 
 the proof of  of Proposition \ref{bv_gffree1}. Define now,  
\be \label{idea:HO}
  \tilde v (x) =\begin{cases}
   & \frac 1 {| \eta_1\, \eta_2\, \eta_3 | }    \sum _{m= 1} ^{| \eta_1\, \eta_2\, \eta_3 |} \,    v^ {[m]}(x),  \quad \text {for }\  x\in \Om_a\, , \\
   & v (x), \quad \text {for }\  x\in \Om_* \, . 
 \end{cases}\ee
Then tracing back the definition of $\Vac $ at the elements next to the interface $\GG$ and the proof of Proposition \ref{bv_gffree1}, we can show that $\tilde v $ is \emph{continuous at the 
interface} $\GG.$ Thus $\tilde v \in H ^ 1 (\Omega ),$ and  is periodic.  Further,
by the Gauss-Green theorem, 
\be  \< D\mathcal{E} _{\,h,  \eta} (y_F), v\>  =\ph'(F\,\eta)\cdot\left\{ \int_{\Om_a }\nabla  \tilde v\,\eta
+
\int_{\Om_*}  \nabla  v\,\eta\right\}=\ph'(F\,\eta)\cdot  \int_{\Om }\nabla  \tilde v\,\eta=0\, ,
 \ee
and the proof is complete. 
\end{proof}

\textbf{Acknowledgements.} Work partially supported by the FP7-REGPOT project \textit{ACMAC: Archimedes Center for Modeling, Analysis and Computation} of the University of Crete. 

\bibliographystyle{abbrv}
\bibliography{CB_AC_bib}

\address{
\noindent
Charalambos\  Makridakis, Dimitrios\  Mitsoudis,\\
Department of Applied Mathematics, University of Crete, 71409 Heraklion-Crete, Greece; and\\
Institute of Applied and Computational Mathematics, FORTH, 71110 Heraklion-Crete, Greece\\
\email{\tt{makr{\it @\,}tem.uoc.gr}}, \email{\tt{dmits{\it @\,}tem.uoc.gr}}\\
\and \\
 Phoebus Rosakis\\
Department of Applied Mathematics, University of Crete, 71409 Heraklion-Crete, Greece\\
\email{\tt{rosakis{\it @\,}tem.uoc.gr}}

}
\end{document}